\documentclass[12pt]{article}
\usepackage{amssymb}
\usepackage{amsmath}
\usepackage[all,arc]{xy}
\usepackage{ytableau}
\usepackage{hyperref}
\usepackage{tikz}
\usepackage[margin=1in]{geometry}

\usepackage{lscape}

\usepackage{amsthm}
\theoremstyle{definition}\newtheorem{remark}{Remark}[section]
\theoremstyle{plain}\newtheorem{propo}[remark]{Proposition}	
\newtheorem{theorem}[remark]{Theorem}
\newtheorem{example}[remark]{Example}
\newtheorem{definition}[remark]{Definition}

\newtheorem{proposition}[remark]{Proposition}
\theoremstyle{definition}

\newcommand{\C}{\mathbb{C}}

\newcommand{\Q}{\mathbb{Q}}

\newcommand{\N}{\mathbb{N}}
\newcommand{\I}{\mathcal{I}^O_{I,J}}
\newcommand{\Ic}[2]{\mathcal{I}^O_{#1,#2}}

\ytableausetup{smalltableaux}

\begin{document}
	
		
		\title{Weingarten Calculus and the {\tt IntHaar} Package for Integrals over Compact Matrix Groups}
		
		\author{Alejandro Ginory and Jongwon Kim}
		\date{}
	\maketitle	
		
		\begin{abstract}
			\noindent In this paper, we present a uniform formula for the integration of polynomials over the unitary, orthogonal, and symplectic groups using Weingarten calculus. From this description, we  further simplify the integration formulas and give several optimizations for various cases. We implemented the optimized Weingarten calculus in {\tt Maple} in a package called {\tt IntHaar} for all three compact groups. Here we will discuss its functions, provide instructions for the package, and produce some examples of computed integrals.
		\end{abstract}
		
		

	
	\section{Introduction}
	Integration over compact Lie groups with respect to the Haar measure is of great importance in a multitude of areas such as physics, statistics, and combinatorics to name a few.  In [1], B. Collins establishes a method of symbolic integration over the unitary group which is later also developed for the orthogonal and the symplectic groups in [2],[3]. This method of integration is now known as \emph{Weingarten calculus} and gives a combinatorial method for evaluating the integrals of polynomials with respect to the Haar measure. In this paper, we use Jack polynomials and twisted Gelfand pairs to give a unified presentation of Weingarten calculus. Using this presentation and combinatorial insights, we produce new optimization algorithms that greatly improve the efficiency and speed of computing these integrals. Finally, we present our implementation of the optimized Weingarten calculus for the unitary, orthogonal, and symplectic groups into a {\tt Maple} package called {\tt IntHaar}. Previously, the only package freely available online for Weingarten calculus dealt solely with the unitary group.
	
	The basic setup is as follows. Given a matrix group $G$, a polynomial function over $G$ is a polynomial in its matrix entries. We are interested in integrals of polynomials over the unitary group $U(d),$ the orthogonal grup $O(d),$ and the symplectic group $Sp(2d).$ In the literature, these integrals are often called the \emph{moments} of the group $G$, [3]. In order to express the integral of a monomial, we use lists to keep track of the indices of the matrix entries as shown below
	
	\begin{align*}
	\mathcal{I}^{U(d)}(I,J,I',J') &= \int_{U(d)} g_{i_1,j_1}...g_{i_n,j_n}\overline{g}_{i'_1,j'_1}...\overline{g}_{i'_n,j'_n}\ dU\\
	\mathcal{I}^{O(d)}(I,J) &=\int_{O(d)} g_{i_1,j_1}...g_{i_n,j_n}\ dO\\
	\mathcal{I}^{Sp(2d)}(I,J) &= \int_{Sp(2d)} g_{i_1,j_1}...g_{i_n,j_n}g_{i_{n+1},j_{n+1}}...g_{i_{n+n},j_{n+n}}\ dSp
	\end{align*}
	where 
	$$I = (i_1, i_2, ..., i_n), I' = (i'_1, i'_2, ..., i'_n), J = (j_1, j_2, ... j_n), J' = (j'_1, j'_2, ..., j'_n).$$
	For example, $I = (1,2,2,1)$ and $J=(1,3,1,3)$ gives 
	$$\int_{O(d)} g_{1,1}g_{2,3}g_{2,1}g_{1,3}\ dO$$
	and similarly for unitary group, which involves complex conjugates, $I = (1,2,2,4)$, $J = (1,3,3,5)$, $I' = (1,2,2,5)$ and $J'=(1,3,3,4)$ gives
	$$\int_{U(d)} |g_{1,1}|^2|g_{2,3}|^4 g_{4,5} \overline{g}_{5,4}\ dU.$$
	
	Using the method of Weingarten calculus, we use the lists and the so-called \emph{Weingarten function} (see \S 2)  to explicitly compute the integrals. In [1], [2], [3], and [7], the authors give various formulas for the Weingarten function that we summarize in a single formula depending on a parameter $\alpha,$ which is the so-called Jack parameter, corresponding to Jack polynomials $J^{(\alpha)}_\lambda$ (and a linear character on the symmetric group). Jack polynomials themselves have been extensively studied and can be computed very efficiently. The uniform formula makes apparent several methods of optimization that can be applied to \emph{all} of the groups in question. In particular, section 4 and section 5 give parallel reduction methods that were discovered as a result of the uniformity of Weingarten calculus. There are many optimizations worked out for the unitary case, for example see [9], while there is not much done for either the orthogonal or symplectic case. Here we present new algorithms that help compute integrals of special cases of monomials efficiently for all three compact groups.
	
	In 2011, Z. Puchala and J. Miszczak distributed a Mathematica package {\tt intU} which computes integrals of functions over the unitary group. Our implementation of Weingarten calculus is called {\tt IntHaar} and, in this work, we describe the functions in the package. In contrast to {\tt intU}, the main functions of {\tt IntHaar} compute integrals as rational functions in $d$ over all aformentioned compact groups, hence we do not write the argument $d$ in our formulas when there is no confusion. The package {\tt IntHaar} extensively uses J. Stembridge's package SF2.4v for symmetric functions and partitions [10]. In particular, SF2.4v is useful in manipulations and computations involving the Jack polynomials. 
	
	This paper is organized as follows. Section 2 discusses preliminaries involving partitions, zonal polynomials, and twisted Gelfand pairs. All of this material can be found in Macdonald's wonderful book on symmetric functions [6]. Section 3 summarizes the integral formulas in a uniform fashion involving Jack polynomials.	Section 4 presents integral formulas and optimizations for the unitary group. In section 5, we give analogous formulas and optimizations for the orthogonal group. In addition, we provide a combinatorial method of reduction based on the Young lattice for a special case. Many of the methods used in the orthogonal case can be directly adapted to the symplectic case, therefore we do not have a separate section for the symplectic group. The main difference between the orthogonal and symplectic cases involves keeping track of the signs of the permutations involved in the various sums. Finally, in sections 6 and 7 we describe the {\tt IntHaar} package in detail and give some examples.
	
	\bigskip
	
	\noindent\textbf{Acknowledgments}  It is our pleasure to thank Siddhartha Sahi for suggesting the implementation of Weingarten calculus in {\tt Maple} and for extensive discussions and advice on this paper. 
	
	\bigskip 
	
	\section{Preliminaries}
	
	We will first quickly review some basics about partitions. A \emph{partition} $\lambda = (\lambda_1, \lambda_2, ..., \lambda_l)$ of $n\in\N,$ $n>0,$ is a sequence of positive integers such that  $\sum_i \lambda_i = n$ and $\lambda_1\geq \lambda_2\geq\cdots\geq \lambda_{l}.$ A {\it Young diagram} of a partition $\lambda$ is a diagram: a left aligned collection of boxes whose $i$th row has exactly $\lambda_i$ many boxes. For example, for $\lambda = (4,3,1)$ has the Young diagram
	
	\ytableausetup{smalltableaux}
	
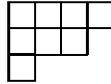
	\begin{center}
		\begin{ytableau}
			\ \ & \ \ &	\ \ & \ \ \cr
			\ \ & \ \ & \ \ \cr
			\ \ \cr
		\end{ytableau}
	\end{center}
	
	Given a partition $\lambda,$ we define the \emph{conjugate} $\lambda'$ to be the partition whose $i$th row equals the number of boxes in the $i$th column of the Young diagram of $\lambda.$ The Young diagram of $\lambda'$ is simply the transpose of the Young diagram of $\lambda.$ For example, for $\lambda=(4,3,1)$ as above, the conjugate is $\lambda' = (3,2,2,1)$ and its Young diagram is
	
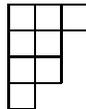
\begin{center}
		\begin{ytableau}
			\ \ & \ \ &	\ \ \cr
			\ \ & \ \  \cr
			\ \ & \ \ \cr
			\ \
		\end{ytableau}
	\end{center}
	
	For each $i$, multiplicity of $i$ in $\lambda$ is denoted by $m_i$. For a partition $\lambda = (\lambda_1, ..., \lambda_l)$ of $n$, we denote by $\ell(\lambda) = l$ the \emph{length} of $\lambda$ and $|\lambda| = n$ the \emph{weight} of $\lambda$. We write $\lambda \vdash n$ if $|\lambda| = n$.  Given a partition $\lambda$, we denote $2\lambda = (2\lambda_1, 2\lambda_2, ...)$ and $\lambda \cup \lambda = (\lambda_1, \lambda_1, \lambda_2, \lambda_2, ...)$, partitions of $2n$. For the trivial partition $(1,1,...,1)$, we write $1^n$. Let $\lambda,\mu$ be partitions of $n,$ then we introduce a partial order called the \emph{dominance order} by declaring $\lambda\ge \mu$ if and only if
	$$\sum_{i=1}^{k}{\lambda_i}\ge \sum_{i=1}^{k}{\mu_i}$$
	for all $k\geq 1.$
	
	\bigskip
	
	Now, we recall some facts regarding zonal polynomials. If $G$ is a finite group, we denote by $\C[G]$ the \emph{group algebra} of $G$ which defined to be the algebra with basis $\{g\}_{g\in G}$ and multiplication on the basis $gh=g\cdot h$ where $\cdot$ is the multiplication of the group $G.$ The multiplication is then extended linearly to the rest of the vector space. Another way of realizing the group algebra is as the space of functions $C(G)=\{f:G\to \C\}$ equipped with the convolution product
	$$f*g(x)=\sum_{y\in G}{f(xy)g(y^{-1})},$$
	for $x\in G.$ We leave it to the reader to verify that $\C[G]\cong C(G)$ as algebras via the map given on the basis of $\C[G]$
	$$g\mapsto \delta_{g}$$
	where $\delta_g(h)=0$ when $g\neq h$ and $\delta_g(g)=1.$ We will frequently identify both spaces  using the above algebra isomorphism. 
	
	We assume familiarity with basic representation theory of finite groups in what follows. Let $G$ be a group, $K$ a subgroup, and $\varphi:K\to\C^*$ a homomorphism, then the triple $(G,K,\varphi)$ is called a \emph{twisted Gelfand pair} if $e_{K}^\varphi\C[G]e_{K}^\varphi$  is a commutative algebra, where $\C[G]$ is the group algebra of $G$ and 
	$$e_{K}^\varphi=\frac{1}{|K|}\sum_{k\in K}{\varphi(k^{-1})k}.$$
	When $\varphi$ is trivial, we omit $\varphi$ and call $(G,K)$ simply a \emph{Gelfand pair}. Let us use the notation $V^\varphi=\C[G]e_{K}^\varphi$ and 
	$$\widehat{V}^\varphi=\{\lambda: \text{Hom}_G(V_\lambda,V^\varphi)\neq 0\}$$
	where $\lambda$ ranges over the equivalence classes of irreducible $G$-representations and $V_{\lambda}$ is a representative of this equivalence class. We restrict our attention to finite groups $G$ unless stated otherwise. Consider the so-called $\varphi$-\emph{zonal spherical functions}
	$$\omega_\lambda^\varphi=\overline{\chi}_\lambda e_K^\varphi$$
	where $\lambda\in \widehat{V}^\varphi$ and $\chi_\lambda$ is its character. As before, if $\varphi$ is trivial, then these functions are called \emph{zonal spherical functions}. By the orthogonality of characters, the $\omega_\lambda$ are orthogonal and can be rescaled so that they form a family of orthogonal idempotents $\widetilde{\omega}_\lambda.$ The subspace of $\C[G]$ generated by the $\omega_\lambda,$ $\lambda\in \widehat{V}^\varphi,$ is an algebra isomorphic to $C^\varphi(K\backslash G / K),$ the convolution algebra of functions $f:G\to \C$ satisfying
	$$f(kg)=f(gk)=\varphi(k^{-1})f(g).$$ It is clear that the elements of $C^\varphi(K\backslash G / K)$ depend only on their values on a set of double coset representatives and so there is a bijection between the $\widehat{V}^\varphi$ and the set of double cosets of $K.$ Fix such a bijection and identify both sets. For any $g\in G,$ we say that $\lambda$ is the \emph{coset-type} of $g$ if $KgK$ is the double coset of $K$ corresponding to $\lambda\in \widehat{V}^\varphi.$
	
	Suppose that the double cosets $K\backslash G/ K$ are indexed by partitions $\lambda$ of $n,$ then, for each $\lambda\in \widehat{V}^\varphi,$ we can define the symmetric polynomial
	$$Z_{\lambda}^\varphi=\frac{1}{|K|}\sum_{g\in G}{\omega_{\lambda}(g)p_{g}}$$
	where $p_{g}=p_{\mu}=p_{\mu_1}p_{\mu_2}\cdots p_{\mu_{\ell}},$ $\mu=(\mu_1,\ldots,\mu_{\ell})$ is the coset-type of $g,$ and 
	$$p_i=\sum_{j}{x_j^i}$$
	are the \emph{power sums}. It is worth mentioning here that the power sums form a basis of symmetric polynomials over $\Q,$ [6]. The polynomials $Z_\lambda$ are called the \emph{zonal polynomials} of the twisted Gelfand pair $(G,K,\varphi).$
	
	The three twisted Gelfand pairs of interest to us are
	\begin{enumerate}
		\item $(S_n\times S_n,S_n)$ where $S_n$ is the diagonal subgroup
		$$S_n\hookrightarrow S_n\times S_n$$
		$$\sigma\mapsto (\sigma,\sigma),$$
		\item $(S_{2n},H_{n}),$ where $H_n$ is the \emph{hyperoctahedral group} which we define to be the centralizer of the involution
		$$(12)(34)\cdots (2n-1,2n)$$
		in $S_{2n},$ and
		\item $(S_{2n},H_n,sgn|_{H_n})$ where $sgn$ is the sign character.
	\end{enumerate}
	We will refer to the twisted Gelfand pairs in $1,2,$ and $3$ by $(G_\alpha,K_\alpha)$ for $\alpha=1,2,1/2,$ respectively. Each compact group $U(d),$ $O(d),$ and $Sp(2d)$ is associated with a twisted Gelfand pair $(G_\alpha,K_\alpha,\varphi_\alpha)$ for $\alpha=1,2,1/2,$ respectively. 
	
	In all three cases, the characters of the group $G$ in $(G,K,\varphi)$ are indexed by partitions (see [5]), as are the coset-types. We will explain how to assign partitions to the coset-types, as it will be useful for some of the optimizations we discuss later. Given a permutation $\sigma \in S_n$, we can write $\sigma$ as a product of disjoint cycles $\sigma_1\sigma_2...\sigma_l$ where $|\sigma_1| \ge |\sigma_2| \ge ... \ge |\sigma_l|$ and $1$-cycles are included. So $\sigma$ corresponds to a partition $(|\sigma_1|, |\sigma_2|, ..., |\sigma_l|)$ of $n$. For example, $\sigma = (124)(35)(69)(7)(8) \in S_9$ corresponds to $(3,2,2,1,1)$. Moreover, two elements in $S_n$ are conjugate if and only if they have the same cycle-type. In the Gelfand pair $(S_n\times S_n,S_n),$ it is straightforward to check that the set $\{(1,\sigma):\sigma\in S_n\}\cong S_n$ is a complete set of left coest representatives of $S_n.$ A subset of these will give us a set of double coset representatives. Indeed, if $(1,\sigma)$ and $(1,\tau)$ lie in the same double coset, then there is are $\gamma,\gamma'\in S_n,$ such that 
	$$(\gamma\gamma',\gamma\sigma\gamma')=(1,\tau).$$ It follows that $\gamma'=\gamma^{-1}$ and so $\tau$ is a conjugate of $\sigma.$ Conversely, if two elements $\sigma,\tau$ are conjugate, then $(1,\sigma)$ and $(1,\tau)$ are in the same double coset. It follows that the double cosets are given by the conjugacy classes in $S_n.$ In particular, to each coset-type we assign the cycle-type of $\sigma$ for any representative of the form $(1,\sigma).$ 
	
	For the double cosets of $H_n$ in $S_{2n},$ we also assign partitions as follows. For each permutation $\sigma \in S_{2n}$, we consider the graph $\Gamma(\sigma)$ of  vertices $\{1,2,...,2n-1, 2n\}$ with edges $\{(2i-1,2i), (\sigma(2i), \sigma(2i)) \mid 1\le i \le n\}$. Let $\Gamma(\sigma)_i,$ $i=1,\ldots,l,$ be the connected components of $\Gamma(\sigma)$ indexed so that $\left|\Gamma(\sigma)_1\right| \ge \left|\Gamma(\sigma)_2\right| \ge ... \ge \left|\Gamma(\sigma)_l\right|,$ where $|\Gamma(\sigma)_i|$ is the number of edges in $\Gamma(\sigma)_i.$ The {\it coset-type} of $\sigma$ is the partition $\left(|\Gamma(\sigma)_1|, |\Gamma(\sigma)_2|, ..., 	|\Gamma(\sigma)_l|\right).$ For example, a permutation $(1,3,4)(2,5)(6) \in S_6$ has a coset-type $(3)$ whereas $(1,2,4)(5,6)(3)$ has $(2,1)$. Two permutations $w, w_1 \in S_{2n}$ have the same coset-type if and only if $w_1 \in H_n w H_n$. The coset-type of $w \in S_n$ is denoted $\rho(w)$ [6]. 
	
	Also useful is a description of left cosets of $H_n.$ The set of \emph{pair partitions} of $\{1,2,\ldots,2n\},$ which are defined to be partitions of the form
	$$\{\{a_1,b_1\},\ldots,\{a_n,b_n\}\}$$
	where $a_i< b_i$ for $i=1,\ldots,n,$ index the set of left coset representatives of $H_n$ in $S_{2n}.$ The set of left coset representatives, denoted $M_{2n},$ are permutations $a_1b_1a_2b_2\cdots a_nb_n,$ written in one-line notation, corresponding to pair partitions $\{\{a_1,b_1\},\ldots,\{a_n,b_n\}\}.$ More generally, we denote the set of left coset representatives by $R_\alpha,$ $\alpha=1,2,1/2,$ corresponding to $S_n,M_{2n},M_{2n},$ respectively. 
	
	The zonal polynomials of the above twisted Gelfand pairs are given by a well-known family of polynomials called \emph{Jack polynomials} $J_\lambda^{(\alpha)},$ for $\lambda\vdash n$ and $\alpha=1,2,1/2,$ respecitvely. The Jack polynomials are symmetric polynomials over the field $\Q(\alpha),$ where $\alpha$ can be taken as a formal parameter. To define these polynomials we need the following notion. Using the fact that the symmetric polynomials are linearly spanned by the power sums $p_{\lambda}$ for all partitions $\lambda,$ we can define the inner product on symmetric polynomials by
	$$\langle p_\lambda,p_\mu\rangle_\alpha = \alpha^{\ell(\lambda)}z_\lambda\delta_{\lambda\mu}$$
	where $z_\lambda =  \prod_{i\geq 1} i^{m_i}m_i!.$ The Jack polynomials in $n$ variables are the symmetric polynomials characterized by
	\begin{enumerate}
		\item $J_\lambda^{(\alpha)}=m_{\lambda,\lambda}+\sum_{\mu< \lambda}{a_{\lambda\mu}m_{\mu}}$ and
		\item $\langle J_\lambda^{(\alpha)},J_\mu^{(\alpha)}\rangle_\alpha = \delta_{\lambda\mu}$
	\end{enumerate}
	where $m_{\lambda}=\sum_{\nu\in S_n\cdot\lambda}{x^\nu}.$ These polynomials play an important role in the integrals over the compact groups we will be considering.
	
	
	\section{The Weingarten Function and Integral Formulas}
	
	Throughout this section, we will assume that the groups in question have associated dimension $d$ (or $2d$ in the symplectic case)  and that $d\geq n,$ unless otherwise specified. We will address the minor modifications needed in the case $d<n$ at the end of the section. Consider the element 
	$$\Phi^{(\alpha)}=\frac{1}{|K|}\sum_{\lambda\vdash n}{Z_\lambda(1)\widetilde{\omega}_\lambda}$$
	where $\widetilde{\omega}_\lambda=\frac{\chi_\lambda(1)}{|G|}\omega_\lambda$ corresponding to the twisted Gelfand pairs $(S_n\times S_n,S_n),$ $(S_{2n},H_n),$ and $(S_{2n},H_n,sgn|_{H_n})$ for $\alpha=1,2,1/2,$ respectively. This element is known as the \emph{Gram matrix} in [2], [3] and it is the (actual) Gram matrix associated to a certain spanning set of the centralizer algebra of $U(d),$ $O(d),$ or $Sp(2d)$ in $\text{End}(V^{\otimes n})$ where $V$ is the defining representation of the respective compact group. It is worth noting that the spanning set in the case of $U(d)$ is in fact a basis, while in the remaining cases the spanning set is not. For more details on the Gram matrix, see [2],[3].
	
	We define the \emph{Weingarten function} $W^{(\alpha)}$ in all cases to be the inverse of $\Phi^{(\alpha)}$ in $e_{K}^\varphi\C[G]e_{K}^\varphi.$ Explicitly, 
	$$W^{(\alpha)}=|K|\sum_{\lambda\vdash n}\frac{1}{Z_\lambda(1)}\widetilde{\omega}_\lambda$$
	and a direct computation will show that 
	\begin{align*}
	W^{(1)}(\sigma,\tau)&=\frac{1}{n!^2}\sum_{\lambda\vdash n}\frac{\chi_\lambda(1)^2}{s_\lambda(1)}\chi_\lambda(\sigma^{-1}\tau),\\
	W^{(2)}(\sigma)&=\frac{2^n n!}{(2n)!} \sum_{\lambda\vdash n}\frac{\chi_{2\lambda}(1)}{J_\lambda^{(2)}(1)}\omega_{\lambda}(\sigma),\\
	W^{(1/2)}(\sigma)&=\frac{2^n n!}{(2n)!} \sum_{\lambda\vdash n}\frac{\chi_{\lambda\cup\lambda}(1)}{J_\lambda^{(1/2)}(1)}\omega_{\lambda}(\sigma),
	\end{align*}
	for $(S_n\times S_n,S_n),$ $(S_{2n},H_n),$ and $(S_{2n},H_n,sgn|_{H_n}),$ respectively, where the Jack polynomial is taken in $d$ many variables. We will be concerned with summing over left cosets of $K$ in the pair $(G,K),$ therefore in the case $(S_n\times S_n,S_n),$ where the left cosets representatives are
	$$\{(1,\sigma):\sigma \in S_n\},$$
	we will define $W^U(\sigma)=W^U(1,\sigma).$ 
	
	We will now give a uniform description of how to compute integrals of polynomials over the unitary, orthogonal, and symplectic groups with respect to their corresponding Haar measures. To do this, we will slightly change the notation used in the previous sections. In the unitary case, we have matrix elements of the form $g_{ij}$ and $\overline{g}_{ij},$ which is the complex conjugate of $g_{ij}.$ Integrals of monomials of the form 
	$$\int_{U(d)} g_{i_1j_1}\cdots g_{i_mj_m}\overline{g}_{i'_1j'_1}\cdots \overline{g}_{i'_nj'_n}\ dU,$$
	vanish whenever $m\neq n,$ [3]. Similarly, in the orthogonal and symplectic cases, we have matrix elements of the form $g_{ij}$ where integrals of monomials of the form 
	$$\int_{G} g_{i_1j_1}\cdots g_{i_mj_m}\ dG,$$
	vanish whenever $m$ is odd, [3]. 
	
	In the case of the unitary group, monomials of the form
	$$g_{i_1j_1}\cdots g_{i_nj_n}\overline{g}_{i'_1j'_1}\cdots \overline{g}_{i'_nj'_n},$$
	are specified using the lists $I=(i_1,i'_1,\ldots,i_n,i'_n)$ and $J=(j_1,j'_1,\ldots,j_n,j'_n).$
	In the orthogonal case, we specify a monomial
	$$g_{i_1j_1}g_{i_2j_2}\cdots g_{i_{2n}j_{2n}},$$
	by the lists $I=(i_1,i_2,\ldots,i_{2n})$ and $J=(j_1,\ldots,j_{2n}).$ In the symplectic case $Sp(2d),$ we specify a monomial
	$$g_{i_1j_1}g_{i_2j_2}\cdots g_{i_{2n}j_{2n}},$$
	by the lists $I=(i_1,i_2,\ldots,i_{2n})$  $J=(j_1,\ldots,j_{2n})$ where $$i_k,j_k\in \{-d,-d+1,\ldots,-1,1,2,\ldots,d\}.$$ In all cases, the list $I$ can be rearranged so that $i_1\leq i_2\leq \cdots $ and $i_1'\leq i_2'\leq\cdots.$ The groups $G_\alpha$ act (on the right) on their respective groups by permuting indices, in particular, $(\sigma,\tau)\in G_1=S_n\times S_n$ acts in the natural way
	$$(i_1,i'_1,\ldots,i_n,i'_n)\cdot (\sigma,\tau)=(i_{\sigma(1)},i'_{\tau(1)},\ldots,i_{\sigma(n)},i'_{\tau(n)}).$$  
	We also have the action of representatives of left cosets $R_\alpha$ on lists by the rules
	$$(i_1,\ldots,i_{2n})\cdot \{\{a_1,b_1\},\ldots,\{a_n,b_n\}\}=(i_{a_1},i_{b_1},\ldots,i_{a_n},i_{b_n}), \text{ for }\alpha=2,$$
	$$(i_1,\ldots,i_{2n})\cdot \{\{a_1,b_1\},\ldots,\{a_n,b_n\}\}=(i_{a_1},-i_{b_1},\ldots,i_{a_n},-i_{b_n}), \text{ for }\alpha=1/2,$$
	and in the case of $\alpha=1,$ $(1,\sigma)$ acts by the right action of the symmetric group. Note that the action of pair partitions in $M_{2n}$ is precisely their action as permutations (up to sign).
	We define the function 
	$$\delta_{I}(\sigma)=\left\{\begin{matrix}
	(-1)^{neg}& ; &\text{ if } \left(I\cdot\sigma\right)_{2k-1}=\left(I\cdot\sigma\right)_{2k} \text{ for }k=1,\ldots,n\\
	0 & ; &\text{ otherwise}
	\end{matrix}
	\right. ,$$
	where $\sigma\in R_\alpha,$ $\left(J\right)_{k}$ denotes the $k$th entry of $J$ and $neg$ is the number of indices $k\in\{1,\ldots,n\}$ such that $\left(I\cdot\sigma\right)_{2k-1}<0.$ Note that $neg=0$ for $\alpha=1,2.$ 
	
	Let $\mathcal{I}^{(\alpha)}(I,J)$ be the integral of the monomial given by the lists $I,J$ in the case $\alpha,$ then [2] and [3] prove that
	$$\mathcal{I}^{(\alpha)}(I,J)=\sum_{\sigma,\tau\in R_\alpha}{\delta_I(\sigma)\delta_J(\tau)W^{(\alpha)}(\sigma^{-1}\tau)}.$$ 
	We will also use the obvious notation $\mathcal{I}_{I,J}^U,$ $\mathcal{I}_{I,J}^O,$ $\mathcal{I}_{I,J}^{Sp},$ $W^U,$ $W^O,$ and $W^{Sp}$ (that we also used in the introduction). We can write the above formula in a different way. For this, we need the right multiplication action of $G_\alpha$ on $R_\alpha.$ In particular, for $\alpha=1,$ $(\tau_1,\tau_2)\in S_n\times S_n$ acts on $(1,\sigma)$ by $(1,\sigma\tau_2\tau_1^{-1}).$ For $\alpha=2,1/2,$ $\tau\in S_{2n}$ acts on $\{\{i_1,j_1\},\ldots,\{i_n,j_n\}\}\in M_{2n}$ by 
	$$\{\{a_1,b_1\},\ldots,\{a_n,b_n\}\}\mapsto \{\{\tau(a_1),\tau(b_1)\},\ldots,\{\tau(a_n),\tau(b_n)\}\}.$$
	Let $\sigma_I,\tau_J\in R_\alpha$ such that 
	$\delta_I(\sigma_I)=|\delta_J(\tau_J)|=1$ and 
	$$S_{I}=\{\gamma\in G_\alpha:(I\cdot \sigma_I)\cdot \gamma = (I\cdot \sigma_I)\}/K_\alpha$$
	$$S_{J}=\{\gamma\in G_\alpha:(I\cdot \tau_J)\cdot \gamma = (I\cdot \tau_J)\}/K_\alpha$$
	be the `stabilizers` of the lists in question. 
	
	\begin{proposition}
		Using the above notation, we have that
		$$\mathcal{I}^{(\alpha)}(I,J)=\sum_{\sigma\in \sigma_IS_I,\tau\in \tau_JS_J}{\delta_I(\sigma)\delta_J(\tau)W^{(\alpha)}(\sigma^{-1}\tau)}.$$
		In particular, when $\alpha=1,2$ we have
		$$\mathcal{I}^{(\alpha)}(I,J)=\sum_{\sigma\in \sigma_IS_I,\tau\in \tau_JS_J}{W^{(\alpha)}(\sigma^{-1}\tau)}.$$
	\end{proposition}
	
	\begin{proof}
		Suppose that for our list $I,$ we have $0<i_1\leq i_2\leq \cdots \leq i_n,$ then $\delta_I(\sigma)=1$ or $0$ in all cases because the pair partitions are defined so that $a_i<b_i$ (see [2]). If $\sigma_1$ and $\sigma_2$ satisfy $\delta_I(\sigma_j)=1$ for $j=1,2,$ then in all cases $\sigma_1^{-1}\sigma_2$ stabilizes $I\cdot \sigma_1.$ If we fix $\sigma_I$ so that $I\cdot \sigma_1=I\cdot \sigma_I,$ then $\sigma_I\gamma_1K_\alpha = \sigma_I\gamma_2 K_\alpha$ if and ony if $\gamma_1K_\alpha = \gamma_2 K_\alpha.$ Similar reasoning holds for the list $J.$ The result then follows immediately. 
	\end{proof}
	
	Finally, in the more general case where $d<n,$ the formulas above hold with the only modification being that in the definition of the Weingarten function, the sum ranges over partitions $\lambda$ such that $\ell(\lambda)\leq d.$ 
	
	\section{Integrals over the Unitary Group}
	
	\subsection{Normalization of Lists}
	In computing the Weingarten function in the case $\alpha=1,$ i.e., the unitary group $U(d)$ case for a fixed $d\geq 1,$ we can take advantage of various facts in order to simplify the form of the lists in question. Let us revert to the original format for the lists introduced in the introduction
	$$I=(i_1,\ldots,i_n), I'=(i_1',\ldots,i_n')$$
	$$J=(j_1,\ldots,j_n), J'=(j_1',\ldots,j_n')$$
	and use the notation
	$$g_{I,J}\overline{g}_{I',J'}:=g_{i_1j_1}\cdots g_{i_nj_n}\overline{g}_{i'_1j'_1}\cdots \overline{g}_{i'_nj'_n}.$$
	Since we are only interested in the integrals of these monomials, we will be considering transformations of the lists $I\to I_0, I'\to I_0',$ $J\to J_0,$ $J'\to J_0'$ that keep the integral invariant, i.e.,
	$$\mathcal{I}^{U(d)}(I,J,I',J')=\mathcal{I}^{U(d)}(I_0,J_0,I'_0,J'_0).$$
	By the commutativity of the variables, we can assume that $I$ is in the following order
	$$(i_1,\ldots,i_{\lambda_1},i_{\lambda_1+1},\ldots,i_{\lambda_1+ \lambda_2},\ldots, i_{n})$$
	where $\lambda_1\geq \lambda_2\geq \cdots \geq \lambda_\ell$ and 
	$$i_1=\cdots=i_{\lambda_1}$$
	$$i_{\lambda_1+1}=\cdots=i_{\lambda_1 + \lambda_2}$$
	$$\vdots$$
	$$i_{n - \lambda_{\ell}+1}=\cdots=i_{n}.$$
	We can do the same for $I'$ but not necessarily for $J,J'.$ Notice that the partition $\lambda$ is unique and independent of the actual values of the $i_k.$
	
	\begin{definition}
		For any list $I=(i_1,\ldots,i_n),$ we call the partition $\lambda$ described above as the \emph{multiplicity partition} of $I.$
	\end{definition}
	By the invariance of the Haar measure under left multiplication, we can multiply $g_{I,J}\overline{g}_{I',J'}$ by permutation matrices so that the resulting monomial $g_{I_0,J_0}\overline{g}_{I'_0,J'_0}$ has list $I_0$ of the form 
	$$(\underbrace{1,\ldots,1}_{\lambda_1},\underbrace{2,\ldots,2}_{\lambda_2},\ldots,\underbrace{\ell,\ldots,\ell}_{\lambda_\ell}).$$
	Indeed, after normalizing in this way and rearranging $I_0'$ in ascending order, we can achieve $I_0=I_0'$ whenever $\mathcal{I}^{U(d)}(I,J,I',J')\neq 0.$ 
	
	The exact same normalization is not possible for the lists $J$ and $J'$ in the general case. Nonetheless, by commuting variables we can arrange the \emph{blocks}
	$$J^1=(j_1,\ldots,j_{\lambda_1})$$
	$$J^2=(j_{\lambda_1+1},\ldots,j_{\lambda_1 + \lambda_2})$$
	$$\vdots$$
	$$J^\ell=(j_{\lambda_1 + \cdots + \lambda_{\ell-1}+1},\ldots,j_{n})$$
	in order with respect to their respective multiplicity partitions $\mu^k,$ $k=1,\ldots,\ell.$ The same can be done to $J'.$ Writing the lists in this form shows that we can work with a single partition $\lambda$ and a pair of lists $J,J'$ instead of 4 lists. 
	
	Another way to represent the same data is by first normalizing $I,I'$ as above and then finding a permutation $\sigma\in S_n$ so that $\sigma J$ is in the form
	$$(\underbrace{1,\ldots,1}_{\mu_1},\underbrace{2,\ldots,2}_{\mu_2},\ldots,\underbrace{\ell,\ldots,\ell}_{\mu_{\ell'}})$$
	where $\mu$ is the multiplicity partition of $J.$ We also find $\sigma'\in S_n$ to put $J'$ in the same form. In the case that the integral of the monomial in nonzero, the multiplicity partitions for $J$ and $J'$ will coincide. We can recover the integral from $\lambda,\mu,\sigma,\sigma'.$ 
	
	\subsection{Reduction to Single Coset Case}
	In the case $\alpha=1,$ it is easy to see that $S_I$ can be identified with the parabolic subgroup $1\times S'_I$ where $S'_I$ is the stabilizer of the list $(i_{\sigma_I(1)}',\ldots,i_{\sigma_I(n)}').$ A similar identification can be done for $J.$ Therefore, we will write $S_I\subseteq S_n$ where we make the above identification.
	
	If we set $\gamma=\sigma_I^{-1}\tau_J,$ we can write the integral formula as
	$$\mathcal{I}^U(I,J)=\sum_{\sigma\in S_I,\tau\in S_J}{W^U(\sigma\gamma\tau)}.$$	
	Since $W^U$ is a linear combination of characters, we will compute $\chi_\lambda(\sigma)$ for $\sigma\in S_I\gamma S_J.$ Indeed, if we define the idempotents
	$$e_I=\frac{1}{|S_I|}\sum_{\sigma\in S_I}\sigma \text{ and } e_J=\frac{1}{|S_J|}\sum_{\sigma\in S_J}\sigma$$
	in $\C[S_n],$ then $$e_I\chi_\lambda e_J(\gamma)=\frac{1}{|S_I|\cdot |S_J|}\sum_{\sigma\in S_I,\tau\in S_J}{\chi_{\lambda}(\sigma\gamma\tau)},$$
	where $e_I\chi_\lambda e_J$ is regarded as a function on $S_n$ in the standard way.	Interchanging the order of summation, the integral formula has the form
	\begin{align*}
	\mathcal{I}^U(I,J)&=\frac{1}{n!^2}\sum_{\lambda\vdash n}\frac{\chi_\lambda(1)^2}{s_\lambda(1)}\sum_{\sigma\in S_I,\tau\in S_J}{\chi_\lambda(\sigma\gamma\tau)}\\
	&=\frac{|S_I|\cdot |S_J|}{n!^2}\sum_{\lambda\vdash n}\frac{\chi_\lambda(1)^2}{s_\lambda(1)}{e_I\chi_\lambda e_J(\gamma)}.
	\end{align*}
	Note also that 
	$$\left(\chi_\lambda e_I\right)\left(\chi_\lambda e_J\right)=\frac{n!}{\chi_{\lambda}(1)}e_I \chi_\lambda e_J,$$
	and so it suffices to compute $\chi_\lambda e_I$ and $\chi_\lambda e_J$ to find $\mathcal{I}^U(I,J).$ We therefore have the following way of writing the integral formula.
	
	\begin{proposition}
		Using the notation described above,
		$$\mathcal{I}^U(I,J)=\frac{|S_I|\cdot |S_J|}{n!}\left(\sum_{\lambda\vdash n}\sqrt{\frac{\chi_\lambda(1)}{s_\lambda(1)}}{\widetilde{\chi}_\lambda e_I}\right)\left(\sum_{\lambda\vdash n}\sqrt{\frac{\chi_\lambda(1)}{s_\lambda(1)}}{\widetilde{\chi}_\lambda e_J}\right)(\gamma),$$
		where $\widetilde{\chi}_\lambda$ are the central idempotents $\frac{\chi_{\lambda}(1)}{n!}\chi_\lambda.$ 
	\end{proposition}
	
	The above proposition shows that it is sufficient to compute all $\chi_\lambda e_I$ in order to compute the integrals. We can further reduce this problem to computing the size cycle-types inside left cosets of so-called parabolic subgroups of $S_n,$ that is, subgroups of the form $S_{n_1}\times \cdots S_{n_\ell}.$ 
	
	\subsection{List Reduction Algorithm}
	The Weingarten formula is easier to compute whenever the stabilizer subgroups $S_I,S_J$ are as small as possible. In the extreme case, $S_I=\{e\}$ precisely when the set of indices $\{i_1,\ldots,i_n\}=\{1,\ldots,n\}.$ We will describe an algorithm that breaks up the integral formula  
	$$\mathcal{I}^U(I,J)=\sum_{j=1}^{k}\mathcal{I}^U(I_j,J)$$ 
	into a sum of integrals $\mathcal{I}_{I_j,J}^U$ with smaller stabilizer groups $S_{I_j}.$ 
	
	Consider a list $I=(i_1,i'_1,\ldots,i_n,i'_n)$ and for simplicity we can assume that $i_1=1.$ Let $p_1,\cdots, p_m$ be all of the indices such that $i'_{p_k}=1.$ We can partition the permutations $\sigma\in S_n$ into the sets
	$$A_k=\{\sigma\in S_n:\sigma(1)=p_k\}.$$
	We can write the integral formula as
	$$\mathcal{I}^U(I,J)=\sum_{k=1}^{m}\sum_{\substack{\sigma\in \sigma_IS_I\cap A_k,\\\tau\in \tau_JS_J}}{W^U(\sigma^{-1}\tau)}.$$ The set of permutations $\sigma_IS_I\cap A_k$ coincides with the set of permutations $\sigma$ that act by
	$$(x,i'_1,\ldots,i_{p_k},x,\ldots,i_n,i'_n)\mapsto (x,x,\ldots)$$
	and satisfies $\delta_I(\sigma)=1,$ where $x$ is a variable. If $k>1,$ then we can replace $x$ by an unused integer in $\{1,\ldots,n\},$ which we can assume without loss of generality is $n.$ It follows that $\sigma_IS_I\cap A_k$ is the set of permutations that transforms $(i_1',\ldots,n,\ldots,i'_n)$ into $(n,i_2,\ldots,i_n).$ The sum of the Weingarten function over $\sigma_IS_I\cap A_k$ and $\tau_JS_J$ is given by the integral $\mathcal{I}_{I_k,J}^U$ where
	$$I_j=(n,i'_2,\ldots,i_{p_k},n,\ldots,i_n,i'_n),$$ and so
	$$\mathcal{I}^U(I,J)=\sum_{k=1}^{m}\mathcal{I}^U(I_k,J).$$ 
	Continuing in the fashion, we can write $\mathcal{I}^U(I,J)$ as a sum of integrals $\mathcal{I}^U(I_0,J),$ where the set of indices in $I_0$ is equal to $\{1,\ldots,n\}.$ 
	
	\begin{remark}
		Whenever two pairs of lists $I_0,J_0$ and $I_1,J_1$ have $I_0=I_1=(1,2,\ldots,n),$ it is clear that their integrals are equal whenever the multiplicity partitions of the lists $J_0$ and $J_1$ are the same.
	\end{remark}
	
	\subsection{Graph Presentation of Lists}
	
	It is important to know when two different pairs of lists $I_0,J_0$ and $I_1,J_1$ give the same integral. By using the normalization discussed in section 4.1 and representing both lists $I$ and 
	$J$ as an $\ell$-partite graph $\mathcal{G}_{I,J}.$ The $k$th part has vertices $\mu^k_{\lambda_1+\cdots+\lambda_{k-1}+1},\ldots,\mu^k_{\lambda_1+\cdots+\lambda_k}$ (each taken distinctly) for $k=1,\ldots,\ell.$ The edges are given by $\{\mu^{k_1}_{i_1},\mu^{k_2}_{i_2}\}$ whenever the $\mu^{k_1}_{i_1}$th entry of $J^{k_1}$ is equal to the $\mu^{k_2}_{i_2}$th entry of $J^{k_2}.$ Any such graph $\mathcal{G}_{I,J}$ can have at most $n$ connected components and that every connected component is a complete graph.
	
	\begin{propo}
		Using the reduction algorithm, we can assume that $I_0=I_1=(1,2,\ldots,n)$ and $J_1,J_2$ are arbitrary lists. The graphs $\mathcal{G}(I_k,J_k),$ for $k=0,1,$ are isomorphic if and only if $\mathcal{I}^U(I_0,J_0)=\mathcal{I}^U(I_1,J_1).$
	\end{propo}
	
	\begin{proof}
		The graphs are disjoint unions of complete graphs. They are isomorphic if and only if the multiplicity partitions of $J_0$ and $J_1$ are the same. After renumbering the lists, it is clear that the integrals are equal.
	\end{proof}	
	
	In the other direction, any $\ell$-partite graph $\mathcal{G}$ with at most $n$ connected components and complete connected components can have its vertices labeled by the parts of a partition $\lambda\vdash n$ so that there is a pair of lists $(I,J)$ satisfying $\mathcal{G}=\mathcal{G}(I,J).$ A similar construction can be done for the orthogonal case, as will be discussed in the next section.
	\bigskip
	
	\section{Integrals over the Orthogonal Group}
	\subsection{Normalization of Lists}
	In this section, we work with the formula
	\begin{align*}
	\I = \sum_{\substack{\sigma \in S_I, \tau \in S_J}} W^{O}(\tau^{-1}\sigma)
	\end{align*}
	In the case $\alpha = 2$, we may normalized the lists as we did in the case $\alpha = 1$. As introduced already, lists
	$$I = (i_1, \ldots, i_{2n}), J = (j_1,\ldots, j_{2n})$$
	will represent the monomial
	$$g_{I,J} = g_{i_1,j_1}\cdots g_{i_{2n},j_{2n}}$$
	We normalize these lists to consider classes of integrals. It is evident from the integral formulas that different monomials may have the same integral. Furthermore, we want to tell when two monomials have the same integral. 
	
	As in the unitary case, by the commutativity of the variables, we assume that $I$ is in the following order
	$$(i_1, \ldots, i_{\lambda_1}, i_{\lambda_1+1}, \ldots i_{\lambda_1+\lambda_2}, \ldots, i_{2n})$$
	where $(\lambda_1,\ldots, \lambda_{l})$ is a partition of ${2n}$ with each $\lambda_i$ being even, and 
	$$i_1=\cdots=i_{\lambda_1}$$
	$$i_{\lambda_1+1}=\cdots=i_{\lambda_1+\lambda_2}$$
	$$\vdots$$
	$$i_{2n - \lambda_{\ell}+1}=\cdots=i_{2n}.$$
	Again, $\lambda$ is the multiplicity partition of $I$. 
	
	Since we assume $I$ to be in such form, we cannot do the same for $J$. It is immediate that interchanging $I$ and $J$ will keep the integral invariant, so we may assume that $k$ such that $\lambda_i \ge \mu_i$ for $1\le i \le k$ is maximal, where $\lambda$ is the multiplicity partition of $I$ and $\mu$ is the multiplicity partition of $J$. 
	
	By the invariance of the Haar measure under left multiplication, in particular permutation matrices, we may assume that 
	$$(\underbrace{1,\ldots,1}_{\lambda_1},\underbrace{2,\ldots,2}_{\lambda_2},\ldots,\underbrace{\ell,\ldots,\ell}_{\lambda_\ell}).$$
	
	However, there are many permutations that do the desired work. We may consider {\it blocks} in $J$,
	$$J^1=(j_1,\ldots,j_{\lambda_1})$$
	$$J^2=(j_{\lambda_1+1},\ldots,j_{\lambda_1+\lambda_2})$$
	$$\vdots$$
	$$J^{\ell}=(j_{2n -\lambda_{\ell}+1},\ldots,j_{2n})$$
	So we apply a permutation fixing $I$ that keeps each block $J^i$ increasing. 
	
	\subsection{One-row case}
	
	In the case $\alpha = 1$, when a monomial is taken over elements in one row, or one column, there is a closed formula for the integrals using the Gamma function, [8]. In the orthogonal case, we have a recursive formula that greatly simplifies the calculation. The orthogonal Weingarten only depends on the coset-type of the argument.
	By the the Proposition $3.1$, it suffices to find the size of each coset-class in the stabilizer of $I$. Let $C(\lambda)$ denote the size of coset-class of $M_{2n}$ corresponding to a partition $\lambda$ of $n$.	
	
	Let $I = (1^{2n})$. First, we consider the most trivial case where $J = (1^2, 2^2, \ldots, n^2)$. In this case, $S_J = \{e\}$, that is, the list $J$ is fixed by only the pair partition $\{\{1,2\},\ldots, \{2n-1,2n\}\}$, which is the identity in $S_{2n}$. Hence, the integral is simply
	$$\I = \sum_{\lambda \vdash n} C(\lambda) W^O(\lambda)$$
	
	Now we compute $C(\lambda)$ for each partition $\lambda$ of $n$. Let $\mu \subset \lambda$ be a partition of $n-1$. So in terms of Young's diagram, $\mu$ is attained by removing a box from $\lambda$. Define $\mu^{\lambda}$
	
	\[ \mu^{\lambda} = \begin{cases} 
	2\mu_-m_{\mu_-}, & \text{if }\mu_- > 0\\
	1, & \text{if }\mu_- = 0
	\end{cases}
	\]  where $\mu_-$ is the part $\mu_i$ such that $\lambda_i = \mu_i + 1$ and $m_{\mu_-}$ is the multiplicity of such $\mu_i$ in $\mu$. Then
	\begin{theorem}
		For any partition $\lambda\vdash n,$ using the definition of $\mu^\lambda$ above,
		$$C(\lambda) = \sum_{\substack{\mu \subset \lambda \\ \mu \vdash n-1}} \mu^{\lambda} C(\mu)$$
	\end{theorem}
	\begin{proof}
		Proceed by induction on $n$. When $n=1$, it is immediate. For each $\sigma \in M_{2n}$, expressed as pair partitions $\{\{\sigma(1),\sigma(2)\}, \ldots, \{\sigma(2n-1), \sigma(2n)\}\}$, we associate a graph $\Gamma(\sigma)$ with ``red edges" $\{(2i-1, 2i) \mid 1 \le i \le n\}$ and ``blue edges" $\{(\sigma(2i-1), \sigma(2i)) \mid 1 \le i \le n\}$. 
		
		On the other hand, consider a graph $\Gamma$ of coset-type $\lambda$, where each vertex has one red edge and one blue edge and red edges are $\{(2i-1, 2i)\mid 1 \le i \le n\}$ (Note that a blue edge and a red edge may overlap). Then, simply by reading off blue edges, we obtain a unique $\sigma \in M_{2n}$ whose $\Gamma(\sigma) = \Gamma$. This is possible by the way we defined $M_{2n}$ to be pair partitions satisfying certain order properties. Hence, there is a bijection between $M_{2n}$ and the set of all graphs of above construction.  
		
		Let $\Gamma$ be a graph with $2n$ vertices constructed as above. Remove blue edges from vertices $2n-1$ and $2n$, and connect the vertices who just lost edges with blue. Then the subgraph of the first $2n-2$ vertices has a coset-type $\mu\vdash n-1$. Then to extend $\Gamma$ to a graph of coset-type $\lambda$, where $\mu \subset \lambda$, we choose a connected component $C$ with $2\mu_-$ vertices ($\mu_-$ is exactly the part where $\mu_i +1 = \lambda_i$). If $\mu_- = 0$, then there is only one way to extend to $\lambda$, by adding the blue edge $(2n-1, 2n)$. Otherwise, we can obtain $\lambda$ by removing a blue edge from $C$ and adding two blue edges from two vertices who just lost a blue edge to $2n-1$ and $2n$. There are $\mu_-$ blue edges from $C$, and two ways to add edges, giving us $2\mu_-$ choices. There are $m_{\mu_-}$ connected components of $\mu_-$ vertices, so we indeed get $2\mu_-m_{\mu_-}$ ways to extend to $\lambda$ from $\mu$ when $\mu_- > 0$. Furthermore, we established the bijection between graphs and $M_{2n}$ and the fact that $M_{2(n-1)} \hookrightarrow M_{2n}$. Hence by induction, we have the desired formula. 
	\end{proof}
	
	We express this recursive formula in terms of a weighted Young's lattice, 
	
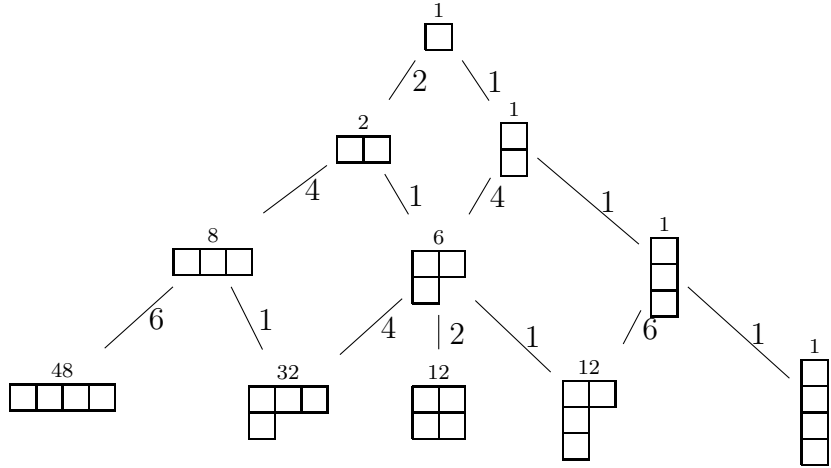
\begin{center}
		\begin{tikzpicture}[shorten >=1pt, auto, node distance=3cm]
		\begin{scope}[every node/.style={}]
		\node (v1) at (0,0) {$\stackrel{1}{\ydiagram{1}}$};	
		\node (v2) at (-1,-1.5) {$\stackrel{2}{\ydiagram{2}}$};
		\node (v3) at (1,-1.5) {$\stackrel{1}{\ydiagram{1,1}}$};
		\node (v4) at (-3,-3) {$\stackrel{8}{\ydiagram{3}}$};
		\node (v5) at (0,-3.2) {$\stackrel{6}{\ydiagram{2,1}}$};
		\node (v6) at (3,-3.2) {$\stackrel{1}{\ydiagram{1,1,1}}$};
		\node (v7) at (-5,-4.8) {$\stackrel{48}{\ydiagram{4}}$};
		\node (v8) at (-2,-5) {$\stackrel{32}{\ydiagram{3,1}}$};
		\node (v9) at (-0,-5) {$\stackrel{12}{\ydiagram{2,2}}$};
		\node (v10) at (2,-5.1) {$\stackrel{12}{\ydiagram{2,1,1}}$};
		\node (v11) at (5,-5) {$\stackrel{1}{\ydiagram{1,1,1,1}}$};
		\end{scope}
		
		\begin{scope}[every edge/.style={draw=black}]
		\draw  (v1) edge node[right, pos=.5]{2} (v2);
		\draw  (v1) edge node[right, pos=.5]{1} (v3);
		\draw  (v2) edge node[right, pos=.5]{4} (v4);
		\draw  (v2) edge node[right, pos=.5]{1} (v5);
		\draw  (v3) edge node[right, pos=.5]{4} (v5);
		\draw  (v3) edge node[right, pos=.5]{1} (v6);
		\draw  (v4) edge node[right, pos=.5]{6} (v7);
		\draw  (v4) edge node[right, pos=.5]{1} (v8);
		\draw  (v5) edge node[right, pos=.5]{4} (v8);
		\draw  (v5) edge node[right, pos=.5]{2} (v9);
		\draw  (v5) edge node[right, pos=.5]{1} (v10);
		\draw  (v6) edge node[right, pos=.5]{6} (v10);
		\draw  (v6) edge node[right, pos=.5]{1} (v11);
		
		\end{scope}
		\end{tikzpicture}
	\end{center}
	
	From this lattice, we obtain $C(\lambda)$ by summing the weight of paths to $\lambda$. 
	
	\begin{remark}
		We observed that for each partition $\lambda$ of $n$, in fact $C(\lambda)$ equals the zonal character, the 	coefficient of $p_{\lambda}$ in power sum expansion of zonal polynomial $Z_{(n)}$.
	\end{remark}
	
	Now given the trivial case, we may consider arbitrary $J$. Since $I$ is fixed to be $(1^{2n})$, we may permute $J$ in any way. So we normalize $J$ in the way that we normalize the list $I$. Then 	
	\begin{theorem}
		Let $\mu$ be the multiplicity partition of $J$. Then 
		$$\I= (\prod_{i=1}^{\ell} (\mu_i-1)!!)\sum_{\lambda \vdash n} C(\lambda) W^O(\lambda)$$
	\end{theorem}
	\begin{proof}		
		Recall that $M_{2n}$ is a complete set of representatives of $S_{2n}/ H_n$ and that two permutations $\sigma, \tau \in S_{2n}$ have the same coset-type if and only if $\sigma \in H_n\tau H_n$. Left multiplication is a transitive action on $S_{2n}$, so $\tau^{-1} M_{2n}$ is still a complete set of representatives of $S_{2n}/H_n$. 
		
		Weingarten function is taken over all $\tau^{-1}M_{2n}$ such that $\tau \in S_J$. The number of $\tau \in S_J$ is $\prod_{i=1}^{\ell} (\mu_i -1)!!$, so the formula follows. 
	\end{proof}
	
	\subsection{Independent Blocks}
	
	In general, we may consider blocks in each list $I$ and $J$. Assume blocks are given by $\mu$. If integers in $J^i$ do not appear in any other blocks, then $J^i$ is {\it independent} in that $\tau \in M_{2n}$ fixing $J^i$ do not act on other blocks at all. Suppose all $J^i$'s are independent. Then for $\sigma \in M_{2n}$ fixing $I$ and $\tau \in M_{2n}$ fixing $J$, we can write 
	$$\sigma = \sigma_1 \cdots \sigma_{\ell}, \quad \tau = \tau_1 \cdots \tau_{\ell}$$
	where $\sigma_i$ and $\tau_i$ act on the $i$th block. It is immediate that all $\sigma_i$'s are disjoint and $\tau_i$'s are also disjoint. Hence, $\tau^{-1}\sigma = (\tau_1^{-1}\sigma_{1}) \cdots (\tau_{\ell}^{-1}\sigma_{\ell})$. Clearly, each $\tau_{i}^{-1}\sigma_{i}$ yields a coset-type corresponding to a partition of $\mu_i/2$. 
	
	Let $C_{\mu_i}(\lambda)$ be the size of coset-class in $M_{\mu_i}$ corresponding to $\lambda$. For each $J^i$, let $m_{i,j}$ be the multiplicity of $j$. Then from the above argument and Theorem 5.3, we get
	\begin{proposition}
		In the above setup, 
		\begin{align*}
		\I = \prod_{i,j}(m_{i,j}-1)!! \sum_{\substack{\lambda_i \vdash \mu_i/2, \\ \lambda = \lambda_1 \cup \cdots \cup \lambda_{\ell}}}  C_{\mu_1}(\lambda_1)\cdots C_{\mu_{\ell}}(\lambda_{\ell}) W^O (\lambda)
		\end{align*}
	\end{proposition}\qed 
	
	\subsection{List Reduction Algorithm}
	Let $C_{I,J}(\lambda)$ denote the number of $\tau^{-1}\sigma$ of coset-type $\lambda$ taken by Weingarten. In case where $J = (1^2,2^2,\ldots,n^2)$, we are computing the coset-classes of the subset of $M_{2n}$ stabilizing $I$. So let $C_{I}(\lambda)$ denote the size of coset-class in $S_I$ corresponding to $\lambda$. ($I$ is not necessarily normalized here.) We start with an example which directly explains the idea and the motivation. 
	
	Consider the monomial given by the lists 
	$$I = (1,1,1,1,1,1,2,2), J = (1,1,1,1,1,2,1,2).$$ 
	It is clear that the issue with this monomial is that $M_{6}\times M_{2}$ is embedded in $S_{8}$ differently for $I$ and $J$. A brute-force calculation will lead to computing the coset type of $225$ permutations, involving conversions from pair partitions to permutations. Consider $\tau \in M_{8}$ fixing the list $J$. As a pair partition, $\tau$ must have $\{6,8\}$. At this point, $\tau$ is determined by a pair-partition of $\{1,2,3,4,5,7\}$. Instead, we may consider
	\begin{align*}
	J_1 &= (3,1,1,1,1,2,3,2), J_2 = (1,3,1,1,1,2,3,2), J_3 = (1,1,3,1,1,2,3,2),\\
	J_4 &= (1,1,1,3,1,2,3,2), J_5 = (1,1,1,1,3,2,3,2)
	\end{align*}
	which represent all possible pairings for the index $7$. It is clear that the lists $J_i$'s are equivalent up to permutation that fixes $I$. Hence, $C_{I,J_i}(\lambda)$ is the same for all $J_i$ and it is enough to consider one $J_i$. 
	
	Applying the same argument recursively, we conclude that 
	$C_{I,J}(\lambda) = 5\cdot3\cdot C_{I,J'}(\lambda)$, where $J' = (1,1,4,4,3,2,3,2) \approx (1,1,2,2,3,4,3,4)$. Permuting $J'$ and $I$ once again, we reduced the problem to consider 
	$$I' = (1,1,1,1,1,2,1,2), J' = (1,1,2,2,3,3,4,4)$$
	which is simply computing the size of each coset-class in $S_{I'}$. So we have
	$$\I = 15 \cdot \Ic{I'}{J'} = 15\sum_{\lambda \vdash n} C_{I'}(\lambda) W^O(\lambda),$$
	where we only need to calculate the coset-type of $15$ permutations. 
	
	We develop an algorithm for general cases. The above example is particularly nice in that we reduced the problem to computing the coset-classes of one $I'$, but in general, we may need to compute for multiple $I'$'s. In certain cases, such as the above example or the one row case, we know what the reduction is explicitly. 
	
	Let $I$ and $J$ be normalized as in section 5.1 with blocks given by $\mu$. Let $m_j$ be the multiplicity of $j$ in $J$. We define an equivalence class on $M_{m_j}$, pair partitions fixing $j$ in $J$. By multiplying permutation matrices, we have an obvious action of $S_{2n}$ on $M_{2n}$ by permuting indices. 
	\begin{definition}
		Let $I,J$ be normalized lists with the multiplicity partition $\mu$. Let $\sigma, \tau \in M_{m_j}$. We say that $\sigma$ and $\tau$ are equivalent if $\sigma$ and $\tau$, as pair partitions, are the same up to a permutation fixing $I$. 
	\end{definition}
	For each $j \in J$, we determine representatives of $M_{2m_j}$ and the size of their class. Let $K_j$ be the list of indices of $j$ in $J$. The equivalence relation is defined by permutations fixing $I$, so all $j$'s in $J_i$ are equivalent. So relabel every $k \in K_j$ by $i$ if $k$ is an index of $j \in J_i$. Now all we need to do is to determine distinct pair partitions of $K_j$ and count their multiplicities. Let $P_j$ be the set of distinct pair partitions of $K_j$. Note that elements in $K_j$ are not actually elements of $M_{2m_j}$ because of the relabeling.
	
	For $\sigma_j \in P_j$, the multiplicity of $\sigma_j$, denoted $m_{\sigma_j}$, can be computed iteratively in the following way. Start by letting $m_{\sigma_j} = 1$. Pick a pair $p = \{p_1,p_2\}$ from $\sigma_j$. We have three cases:
	\begin{enumerate}
		\item $p_1 = p_2$ and there is no more $p_1$ in $\sigma_j$ 
		
		Then we leave $m_{\sigma_j}$ as it is. 
		
		\item $p_1 = p_2$ and there is more $p_1$ in $\sigma_j$ 
		
		Let $m_{p_1}$ denote the multiplicity of $p_1$ in $K_j$. Then $m_{\sigma_j} = \frac{1}{2} m_{\sigma_j}m_{p_1}(m_{p_1}-1)$
		
		\item $p_1 \ne p_2$. 
		Let $m_{p_i}$ denote the multiplicity of $p_i$ in $K_j$. Then $m_{\sigma_j} = m_{\sigma_j}m_{p_1}m_{p_2}$
	\end{enumerate}
	Then we remove $p$ from $\sigma_j$ and $K_j$ and iterate this procedure.  
	
	Iterate this procedure for each $j \in J$. Then by defining $P$ to be the Cartesian product of $P_j$'s, we obtain representatives of equivalent classes we defined on full $M_{2n}$. The multiplicity of each representative is defined naturally by taking the product. There may be some redundancy, so we may further reduce $P$, and properly adjust multiplicities. For each $\sigma \in P$, let $m_{\sigma}$ denote the size of the equivalence class represented by $\sigma$. As we noted above, elements of $P$ are not pair partitions of $\{1,\ldots 2n\}$. So replace each $\sigma \in P$ by an arbitrary pair partition corresponding to $\sigma$. Each pair partition corresponds to a unique permutation of the list $(1,1,2,2,\ldots,n,n)$. For each $\sigma \in P$, define $J_{\sigma}$ to be the corresponding list. Apply some permutation of $S_{2n}$ to revert $J_{\sigma}$ in ascending order and apply the same permutation to $I$ to attain $I_{\sigma}$. Now $J_{\sigma} = (1^2,2^2,\ldots,n^2)$, with $S_{J_{\sigma}} = \{e\}$, so we obtain the formula
	
	\begin{proposition}
		In the above setup,
		$$C_{I,J}(\lambda) = \sum_{\sigma \in P} m_{\sigma}C_{I_{\sigma}} (\lambda).$$
	\end{proposition}
	
	\begin{example}
		Let $I = (1,1,1,1,2,2,2,2,3,3)$, and $J = (1,1,2,2,1,1,1,2,1,2)$. 
		
		First, consider $1 \in J$. Compute all pair-partitions of indices of $1$ in $J$, that is, pair-partitions of $(1,2,5,6,7,9)$. Then by mapping $1,2,$ to $1$, and $5,6,7$ to $2$, and $9$ to $3$, we get three distinct classes represented by  
		$$\{\{1,1\},\{2,2\},\{2,3\}\}, \{\{1,2\},\{1,2\},\{2,3\}\},  \{\{1,2\},\{1,3\},\{2,2\}\}$$
		with multiplicities $3,6,$ and $6$
		
		Similarly for $2 \in J$, we get
		$$\{\{1,1\},\{2,3\}\}, \{\{1,2\},\{1,3\}\}$$
		with multiplicities $1,$ and $2$
		
		By taking the Cartesian product and reducing to the distinct classes, we get six representatives of $M_{2n}$,
		\begin{align*}
		I_{\sigma_1} = (1, 1, 4, 4, 2, 2, 3, 5, 3, 5), \quad m_{\sigma_{1}} = 3 \\
		I_{\sigma_2} = (1, 1, 4, 5, 2, 2, 3, 4, 3, 5), \quad m_{\sigma_2} = 12\\
		I_{\sigma_3} = (1, 2, 4, 4, 1, 2, 3, 5, 3, 5), \quad m_{\sigma_3} = 6\\
		I_{\sigma_4} = (1, 2, 4, 5, 1, 2, 3, 4, 3, 5), \quad m_{\sigma_4} = 12\\
		I_{\sigma_5} = (1, 2, 4, 5, 1, 3, 3, 4, 2, 5), \quad m_{\sigma_5} = 12
		\end{align*}
		
		So the size of each coset-class is 
		\begin{table}[htb]
			\centering
			\begin{tabular}{rrrlrrr}
				& \quad $I_{\sigma_1}$ & $I_{\sigma_2}$ & \quad $I_{\sigma_3}$ & $I_{\sigma_4}$ & $I_{\sigma_5}$ & $I,J$  \\ 
				(5)	:	& 	0 &	 +\quad 48 & +\quad	24 &	+\quad 72 & +\quad	48& =\quad192 \\ 
				(4,1) :	& 	0 &	  +\quad48& +\quad	 12& +\quad	0 & +\quad	24 & =\quad84 \\ 
				(3,2) :		&  12 &  +\quad0& 	 +\quad12& +\quad	36 & +\quad	24 &  =\quad84 \\ 
				(3,1,1) :	& 	 6&	  +\quad12&  +\quad	0& +\quad 	0&	 +\quad	0&  =\quad18 \\ 
				(2,2,1) :	&  	6& 	 +\quad0&  +\quad	6& 	+\quad	 0& +\quad	12&  =\quad24 \\ 
				(2,1,1,1) :	&  	3& 	 +\quad0&  +\quad	0&  +\quad	0& 	+\quad	0& =\quad 3 \\ 
				(1,1,1,1,1) : &  0& +\quad0&  +\quad	0& 	+\quad 	0&  +\quad	0& =\quad 0 \\ 
			\end{tabular}
		\end{table}
	\end{example}
	
	\begin{remark}
		Here we used the permutations fixing $I$ to reduce the problem. But there are other permutations that may help reducing the problem. For example, permutations that permute blocks of the same size in $I$ are essentially just renormalizing $I$, so we must have some invariance here.  
	\end{remark}
	
	\subsection{Graph Presentation of Lists}
	Given $I,J$, where $J$ is completely reduced to $(1^2,\ldots, n^2)$, we may consider $J$ to be normalized as in section 5.1 and consider blocks of size $2$ in $I$. From the results from this section, it is clear that what determines the integral is the number of blocks $I_i = (\alpha, \beta)$, where $\alpha \ne \beta$ and their multiplicities in $I$. 
	
	Let $\lambda = (\lambda_1,\ldots,\lambda_{\ell})$ be the multiplicity partition of $I$. We associate $I$ with the graph $\mathcal{G}(I)$ of $\ell$ vertices with weighted edges, where the edge between $\alpha,\beta$ has weight $N$ if there are $N$ blocks $I_i = (\alpha, \beta)$. Viewing $\lambda$ as a list, $S_{\ell}$ acts on $\lambda$ by permuting indices, and similarly on $\mathcal{G}(I)$. For two lists $I,I'$, we define $\mathcal{G}(I)$ and $\mathcal{G}(I')$ to be equivalent if they are the same up to a permutation fixing $\lambda$. Such a permutation is simply relabeling integers with the same multiplicities, so  we observe that the lists $I$ and $I'$ produce the same integral with the above $J$ if the associated graphs $\mathcal{G}(I)$ and $\mathcal{G}(I')$ are equivalent. 
	
	\section{Program Manual}
	
	In computing partitions, and zonal characters, we used Stembridge's {\tt SF2.4v} package. Below are basic functions used in the {\tt IntHaar} package in {\tt MAPLE}. 
	
	\medskip
	{\bf Helper Functions}:
	
	[{\tt zl($\lambda$::partition)}]: Given partition $\lambda$ of $n$, returns $z_{\lambda}$.
	
	[{\tt CycleType($\sigma$::disjcyc, n)}]:  Given $\sigma \in S_n$, returns its cycle-type. 
	
	[{\tt CosetType($\sigma$::pair partition)}]: Given $\sigma \in M_{2n}$ , returns its coset-type. 
	
	\medskip
	{\bf Unitary Case}:
	
	[{\tt SchurAtOne($\lambda$::partition)}]: Computes the $Z^{(1)}_{\lambda}(1)$.
	
	[{\tt WeingartenU($\lambda$::partition, n)}]: Computes $W^{U}(\lambda, d)$.
	
	[{\tt GetPermU(I, n)}]: Returns all $\sigma \in S_I$. 
	
	[{\tt IUL(I,J)}]: Computes the integral of the monomial in $d$. 
	
	\medskip 
	{\bf Orthogonal Case}
	
	[{\tt ZonalAtOne($\lambda$::partition)}]: Computes $Z^{(2)}_{\lambda}(1)$.
	
	[{\tt ZonalCharacter($\lambda$, $\mu$::partition)}]: Computes $\omega^{\lambda}(\mu)$.
	
	[{\tt GetPermO(I)}]: Returns all $\sigma \in S_I$ 
	
	[{\tt WeingartenO($\lambda$, n)}]: Computes $W^O(\lambda, d)$.
	
	[{\tt IOL(I, J)}] 
	
	\medskip
	{\bf Symplectic Case}
	
	[{\tt GetPermSp(I, n)}]: Returns all $\sigma \in S_I$. 
	
	[{\tt WgSign($\sigma$, $\tau$, n)}]: Computes $(-1)^{neg}$.
	
	[{\tt ISpL(I, J)}]:
	
	\medskip 
	{\bf IntHaar}:
	
	[{\tt IntU(f::polynomial function over $U(d)$)}]
	
	[{\tt IntO(f::polynomial function over $O(d)$)}]
	
	[{\tt IntSp(f::polynomial function over $Sp(2d)$, d)}]
	\bigskip
	
	There are a lot more functions used as helper functions, or for optimization.

	\section{Appendix}
	In the case $\alpha = 2$, we give values of $C_I(\lambda)$ for $n = 10$. In this particular example the computation of each line is considerably less than half a second. For $C_I(\lambda)$ which take more than one second on {\tt MAPLE}, we will provide tables on our website. 
	
	\begin{table}[htb]
		\centering
		\caption{n = 10}
		\bigskip 
		\begin{tabular}{r|rrlrrrr}
			& $(5)$ & $(4,1)$ & $(3,2)$ & $(3,1,1)$ & $(2,2,1)$ & $(2,1,1,1)$ & $(1,1,1,1,1)$ \\ \hline
			$(1,1,1,1,1,1,1,1,1,1)$	& 380 & 240 & 160 & 80 & 60 & 20 & 1 \\
			$(1,1,1,1,1,1,1,1,2,2)$	& 0 & 48 & 0 & 32 & 12 & 12 & 1 \\
			$(1,1,1,1,1,1,1,2,1,2)$	& 48 & 24 & 20 & 6 & 6 & 1 & 0 \\
			$(1,1,1,1,1,1,2,2,2,2)$  & 0 & 0 & 16 & 8 & 12 & 8 & 1 \\
			$(1,1,1,1,1,2,1,2,2,2)$	& 16 & 16 & 4 & 6 & 2 & 1 & 0 \\
			$(1,1,1,2,1,2,1,2,1,2)$  & 24 & 6 & 12 & 0 & 3 & 0 & 0 \\
			$(1,1,1,1,1,1,2,2,3,3)$  & 0 & 0 & 0 & 8 & 0 & 6 & 1 \\
			$(1,1,1,1,1,1,2,3,2,3)$	& 0 & 0 & 8 & 0 & 6 & 1 & 0 \\
			$(1,1,1,1,1,2,1,2,3,3)$	& 0 & 8 & 0 & 4 & 2 & 1 & 0 \\
			$(1,1,1,1,1,2,1,3,2,3)$	& 8 & 4 & 2 & 1 & 0 & 0 & 0 \\
			$(1,1,1,2,1,2,1,3,1,3)$	& 8 & 2 & 4 & 0 & 1 & 0 & 0 \\
			$(1,1,1,1,2,2,2,2,3,3)$	& 0 & 0 & 0 & 0 & 4 & 4 & 1 \\
			$(1, 1, 1, 1, 2, 2, 2, 3, 2, 3)$	& 0 & 0 & 4 & 2 & 2 & 1 & 0 \\
			$(1, 1, 1, 2, 1, 2, 2, 2, 3, 3)$	& 0 & 4 & 0 & 4 & 0 & 1 & 0 \\
			$(1, 1, 1, 2, 1, 2, 2, 3, 2, 3)$	& 4 & 2 & 2 & 0 & 1 & 0 & 0 \\
			$(1, 1, 1, 2, 1, 3, 2, 2, 2, 3)$	& 4 & 4 & 0 & 1 & 0 & 0 & 0 \\
			$(1, 2, 1, 2, 1, 2, 1, 2, 3, 3)$	& 0 & 6 & 0 & 0 & 3 & 0 & 0 \\
			$(1, 2, 1, 2, 1, 2, 1, 3, 2, 3)$	& 6 & 0 & 3 & 0 & 0 & 0 & 0 \\
			$\vdots$
		\end{tabular}
	\end{table}
	

\end{document}